\documentclass[a4paper,reqno,final]{amsart}

\usepackage[bookmarks=false,draft=false,breaklinks,colorlinks]{hyperref}
\usepackage{fullpage}
\usepackage{mathtools,amssymb}
\usepackage{ytableau}
\usepackage{tikz} 
\usepackage{ytableau}
\usepackage{here}
\usepackage{showkeys} 

\numberwithin{equation}{section}
\numberwithin{figure}{section}
\numberwithin{table}{section}
\allowdisplaybreaks
\newenvironment{Ack}%
{\par \vspace{\baselineskip}%
 \noindent \textbf{Acknowledgements.}}%
{\par \vspace{\baselineskip}}

\usepackage{enumitem}
\setenumerate{label=(\arabic*),nosep}
\setitemize{nosep}
\newlist{clist}{enumerate}{1}
\setlist*[clist]{label=(\roman*), nosep}

\usepackage{cleveref}
\crefname{thm}{Theorem}{Theorems}
\crefname{cor}{Corollary}{Corollaries}
\crefname{dfn}{Definition}{Definitions}
\crefname{fct}{Fact}{Facts}
\crefname{lem}{Lemma}{Lemmas}
\crefname{prp}{Proposition}{Propositions}
\crefname{rmk}{Remark}{Remarks}
\crefname{figure}{Figure}{Figures}
\crefname{section}{\S\!}{\S\S\!}
\crefname{subsection}{\S\!}{\S\S\!}
\crefname{subsubsection}{\S\!}{\S\S\!}
\crefname{equation}{equation}{equations}

\usepackage{amsthm}
\theoremstyle{definition}
\newtheorem{thm}{Theorem}[subsection]
\newtheorem{dfn}[thm]{Definition}
\newtheorem{lem}[thm]{Lemma}
\newtheorem{prp}[thm]{Proposition}

\newtheorem{fct}[thm]{Fact}
\newtheorem{rmk}[thm]{Remark}

\newtheorem{q}{Question}

\newcommand{\ol}{\overline}

\newcommand{\ceq}{\coloneqq} 

\newcommand{\srj}{\twoheadrightarrow}
\newcommand{\lto}{\longrightarrow}

\newcommand{\mto}{\mapsto}

\newcommand{\linj}{\lhook\joinrel\longrightarrow}

\newcommand{\ep}{\epsilon}

\newcommand{\thf}{\tfrac{1}{2}}

\newcommand{\bbC}{\mathbb{C}}
\newcommand{\bbN}{\mathbb{N}}
\newcommand{\bbQ}{\mathbb{Q}}
\newcommand{\bbZ}{\mathbb{Z}}
\newcommand{\clP}{\mathcal{P}}
\newcommand{\frS}{\mathfrak{S}}
\newcommand{\Par}{\mathrm{Par}}
\newcommand{\SPar}{\mathrm{SPar}}

\DeclareMathOperator{\spf}{sp}
\DeclareMathOperator{\Pf}{Pf}
\DeclareMathOperator{\GL}{GL}

\DeclareMathOperator{\Sp}{Sp}

\DeclareMathOperator{\SpT}{STab}
\DeclareMathOperator{\QT}{QTab}

\newcommand{\abs}[1]{\left|{#1}\right|}

\newcommand{\br}[1]{\langle#1\rangle}
\newcommand{\sbr}[1]{\left\{#1\right\}}

\newcommand{\trs}[1]{{}^t \! {#1}}

\begin{document}

\title{Intermediate symplectic $Q$-functions}
\author{Shintarou Yanagida}
\date{2022.07.07}
\keywords{Schur's $Q$-functions, symplectic $Q$-functions, symmetric functions, J\'{o}zefiak-Pragacz formula, tableau-sum formula, Lindstr\"{o}m-Gessel-Viennot theorem}
\address{Graduate School of Mathematics, Nagoya University.
 Furocho, Chikusaku, Nagoya, Japan, 464-8602.}
\email{yanagida@math.nagoya-u.ac.jp}
\thanks{The author is supported by supported by 
 JSPS Grants-in-Aid for Scientific Research No.\ 19K03399.}

\begin{abstract}
We introduce an intermediate family of Laurent polynomials between Schur's $Q$-functions and S.\ Okada's symplectic $Q$-functions. It can also be regarded as a $Q$-function analogue of Proctor's intermediate symplectic characters, and is named the family of intermediate symplectic $Q$-functions. We also derive a tableau-sum formula and a J\'ozefiak-Pragacz-type Pfaffian formula of the Laurent polynomials.
\end{abstract}

\maketitle

{\small \tableofcontents}

\setcounter{section}{-1}
\section{Introduction}\label{s:0}

The theme of this brief note is the investigation of an \emph{intermediate} family between symmetric and symplectic polynomials, i.e., a family of polynomials sitting in the intermediate between the two families of polynomials, one of which is invariant under the action of the Weyl group of root system of type $A$, and another is invariant under the action of the Weyl group of type $C$. A typical example of such an intermediate family is Proctor's intermediate symplectic Schur polynomials \cite{P}, which was introduced as the characters of indecomposable representations of intermediate symplectic Lie groups. The resulting polynomial is denoted as 
\begin{align}\label{eq:0:isp}
 \spf^{(k,n-k)}_\lambda(x_1,\dotsc,x_n) \in 
 \bbZ[x_1^{\pm1},\dotsc,x_k^{\pm1},x_{k+1},\dotsc,x_n]^{W_k \times \frS_{n-k}},
\end{align}
where $k$ and $n$ are non-negative integers satisfying $k \le n$, $\lambda=(\lambda_1,\dotsc,\lambda_n)$ is a partition of length $\le n$ (see \cref{ss:0:ntn} below for the terminology), the first group of variables $(x_1,\dotsc,x_k)$ is of type $C$, acted by the Weyl group $W_k$ of type $C$,  and the second group $(x_{k+1},\dotsc,x_n)$ is of type $A$, acted by the symmetric group $\frS_{n-k}$. For the special values of $(k,n-k)$, we have
\[
 \spf^{(k,0)}_\lambda(x_1,\dotsc,x_k) = \spf_\lambda(x_1,\dotsc,x_k), \quad 
 \spf^{(0,n)}_\lambda(x_1,\dotsc,x_n) = s_\lambda(x_1,\dotsc,x_n),
\]
where $\spf_\lambda$ is the symplectic Schur polynomial and $s_\lambda$ is the Schur polynomial, respectively. See \cref{ss:is:s-ss} and \cref{ss:is:is} for the detail.

The first motivation of our study was to introduce a nice $q$- or $t$-analogue of the intermediate symplectic Schur polynomial which would be an intermediate version of the Macdonald polynomials of type $A$ and $C$ \cite{M}. 
Our trial failed in that direction, but remained alive in the direction to find a ``($t=-1$)-analogue''. To spell out, let us recall Schur's $Q$-function $Q_\lambda(x_1,\dotsc,x_n)$ \cite[III.8]{M95}. It is the specialization of the Hall-Littlewood polynomial $P_\lambda(x_1,\dotsc,x_n;t)$ \cite[III]{M95} at $t=-1$: 
\begin{align}\label{eq:0:Q}
 Q_\lambda(x_1,\dotsc,x_n) \ceq P_\lambda(x_1,\dotsc,x_n;-1) \in \bbZ[x_1,\dotsc,x_n]^{\frS_n}.
\end{align}
The specialization $t=-1$ implies that the polynomial $Q_\lambda$ is not zero only if $\lambda=(\lambda_1,\dotsc,\lambda_n)$ is a strict partition, i.e., a strictly decreasing sequence $\lambda_1>\lambda_2>\dotsb$.

The polynomial \eqref{eq:0:Q} is of course an object of type $A$. As for the corresponding object of type $C$, let us give a brief explanation on the recent work of S.\ Okada \cite{OsQ}. There he studied the specialization of Macdonald's zonal spherical polynomial of type $C$ \cite{M71}, denoted as $P^C_\lambda(x_1,\dotsc,x_k;t_s,t_l)$, at $t_s=t_l=-1$. Here we denoted the two parameters by $t_s$ and $t_l$, indicating that they are attached to the $W_k$-orbit of short roots and that of long roots in the root system of type $C$, respectively. The specialized Laurent polynomial is named the symplectic $Q$-function, which we denote as
\[
 Q^C_\lambda(x_1,\dotsc,x_k) \ceq P^C_\lambda(x_1,\dotsc,x_k;-1,-1) \in 
 \bbZ[x_1^{\pm1},\dotsc,x_k^{\pm 1}]^{W_k}.
\]
Here we only need to consider a strict partition $\lambda$ as in the case of \eqref{eq:0:Q}.

Now it is tempting to consider if there is an intermediate polynomial 
\begin{align}\label{eq:0:iQ}
 Q^{(k,n-k)}_\lambda(x_1,\dotsc,x_n) \in 
 \bbZ[x_1^{\pm1},\dotsc,x_k^{\pm1},x_{k+1},\dotsc,x_n]^{W_k \times \frS_{n-k}}
\end{align}
whose special cases recover Schur's and symplectic $Q$-functions as
\[
 Q^{(k,0)}_\lambda(x_1,\dotsc,x_k) = Q^C_\lambda(x_1,\dotsc,x_k), \quad 
 Q^{(0,n)}_\lambda(x_1,\dotsc,x_n) = Q_\lambda(x_1,\dotsc,x_n),
\]
so that it can be regarded as a $Q$-function analogue of the intermediate symplectic Schur polynomial \eqref{eq:0:isp}.

The purpose of this note is to explain that there exists such a family of polynomials \eqref{eq:0:iQ}, which we call \emph{the intermediate symplectic $Q$-polynomials}. A direct definition is given by a natural combination of tableau-sum formulas of Schur's $Q$- and symplectic $Q$-functions (see \cref{prp:iQ:iQ-T}). But we make a detour, and take the infinite-variable version as the definition.

Recall that Schur, Hall-Littlewood, Macdonald polynomials and Schur's Q-functions have their infinite-variable version, or the associated symmetric functions, as developed in \cite{M95}. The symplectic Schur polynomial and symplectic $Q$-function also have their infinite-variable version, realized by symmetric functions \cite{Oisp, OsQ}. We start the discussion to introduce \emph{the intermediate symplectic Schur function} $\spf^I_\lambda(X \mid Y)$, which is the infinite-variable version of the intermediate symplectic Schur polynomial \eqref{eq:0:isp}. It has two families $X=(x_1,x_2,\dotsc)$ and $Y=(y_1,y_2,\dotsc)$ of infinite variables. The $X$-variables correspond to the type $C$ variables, and the $Y$-variables to the type $A$ variables. See \cref{dfn:is:is} for the precise definition. The content in \cref{s:is} is a preliminary of the main \cref{s:iQ} of this note, and also serve complimentary material for Okada's paper \cite{Oisp} on an application of the intermediate symplectic Schur polynomials to the counting of shifted plane partitions of shifted double staircase shape.

After that, we introduce in \cref{dfn:iQ:iQf} \emph{the intermediate symplectic $Q$-function} $Q^I_\lambda(X \mid Y)$ which is the infinite-variable version of \eqref{eq:0:iQ}. We also introduce the skew diagram version $Q^I_{\lambda/\mu}(X \mid Y)$. Then we show in \cref{prp:iQ:iQ-T} that a particular specialization of variables $X$ and $Y$ yields a Laurent polynomial enjoying a tableau-sum formula, which is nothing but the finite-variable intermediate symplectic $Q$-polynomial \eqref{eq:0:iQ}. 

In the main \cref{thm:iQ:iJP}, we give a J\'ozefiak-Pragacz-type Pfaffian formula for the intermediate symplectic $Q$-polynomials \eqref{eq:0:iQ}:
\[
  Q^{(k,n-k)}_{\lambda/\mu}(x_1,\dotsc,x_n) = \Pf
 \begin{bmatrix} M^I_\lambda & N^I_{\lambda,\mu} \\ -\trs{N^I_{\lambda,\mu}} & O \end{bmatrix},
\] 
where $M^I_\lambda \ceq \bigl[ Q^{(k,n-k)}_{(\lambda_i,\lambda_j)}(x_1,\dotsc,x_n) \bigr]_{i,j=1}^l$is the matrix consisting of the intermediate symplectic $Q$-polynomials of two-row diagram $(\lambda_i,\lambda_j)$, and $N^I_{\lambda,\mu} \ceq \bigl[Q^{(k,n-k)}_{(\lambda_i-\mu_{m+1-j})}(x_1,\dotsc,x_n)\bigr]_{1 \le i \le l, 1 \le j \le m}$ consists of those of one-row diagram $(\lambda_i-\mu_{m+1-j})$. Our proof is an application of the Lindstr\"{o}m-Gessel-Viennot theorem \cite{GV,L} to a certain directed graph $\Gamma^{(k,n-k)}$ (see \cref{fig:iQ:LGV}). See the proof of  \cref{thm:iQ:iJP} for the detail.

\subsection{Notation and terminology}\label{ss:0:ntn}

Here is the list of notations used throughout the text.
\begin{enumerate}
\item
We denote by $\bbN \ceq \{0,1,2,\dotsc\}$ the set of non-negative integers.



\item 
For $n \in \bbZ_{>0}$, the symbol $\frS_n$ denotes the symmetric group of degree $n$.

\item
The symbol $[a_{ij}]_{i,j=1}^n$ denotes the square matrix of size $n$ 
with $(i,j)$-entry given by $a_{ij}$.




\end{enumerate}

We also use the terminology of partitions, Young diagram and Young tableaux 
in the sense of \cite[Chap.\ I]{M95}. In particular, we use the following terminology.
\begin{enumerate}[resume]
\item
A partition $\lambda=(\lambda_1,\lambda_2,\dotsc,\lambda_l)$ means a non-increasing finite sequence of non-negative integers. 
We identify two such sequences which differ only by a string of zeros at the end: $(\lambda_1,\dotsc,\lambda_l)=(\lambda_1,\dotsc,\lambda_l,0)=(\lambda_1,\dotsc,\lambda_l,0,\dotsc,0)$.
We denote by $\Par$ the set of all partitions.

\item
For a partition $\lambda=(\lambda_1,\lambda_2,\dotsc,\lambda_l)$, 
the non-zero entries $\lambda_i$ are called the parts of $\lambda$. 
The number of parts is called the length of $\lambda$ and denoted by $\ell(\lambda)$.
We also denote $\abs{\lambda} \ceq \lambda_1+\dotsb+\lambda_l$ and call it the weight of $\lambda$.
We denote by $\Par_n \subset \Par$ the subset of partitions of weight $n$.

\item
We use the abbreviation $(m^n) \ceq (\overbrace{m,m,\dotsc,m}^{\text{$n$ times}})$ for the iterated parts in a partition.

\item \label{i:ntn:strpar}
A partition $\lambda$ is called strict if all the parts are distinct.
We denote by $\SPar \subset \Par$ the subset consisting of all strict partitions.

\item \label{i:ntn:Y}
A partition $\lambda$ is identified with the corresponding Young diagram
\[
 \{(i,j) \in \bbZ^2 \mid 1 \le i \le \ell(\lambda), \ 1 \le j \le \lambda_i\},
\]
which is depicted by replacing the lattice points in $S(\lambda)$ with unit cells.

\item \label{i:ntn:trs}
For a partition $\lambda$, we denote by $\lambda'$ the transpose of $\lambda$.

\item
A skew diagram is a set-theoretic difference $\lambda/\mu \ceq \lambda-\mu$ of the Young diagrams corresponding to two partitions $\lambda$ and $\mu$.
\end{enumerate}

Finally, let us note:
\begin{enumerate}[resume]
\item \label{i:ntn:sf}
We follow the terminology on symmetric functions and symmetric polynomials in \cite[Chap.\ I]{M95}.  In particular, a symmetric polynomial means a finite-variable symmetric polynomial, and a symmetric function means an infinite-variable symmetric ``polynomial''. The precise definition will be briefly reviews in \cref{ss:is:biv}. We also use a non-standard terminology ``Schur's $Q$-polynomial'' to mean the finite-variable version of Schur's $Q$-function. See the paragraph of \eqref{eq:iQ:QAp} for the detail.
\end{enumerate}

\section{Intermediate symplectic Schur functions}\label{s:is}

\subsection{The ring of symmetric functions}\label{ss:is:biv}

Let us recall the ring of symmetric functions \cite[I.2]{M95}. For an infinite sequence $X=(x_1,x_2,\dotsc)$ of commuting independent variables, we denote by $\Lambda(X)$ the ring of symmetric functions with variables $X$ \emph{with coefficients in $\bbQ$}. We can regard it as the space of symmetric polynomials of infinite variables $X$, and roughly express its definition as 
\[
 \Lambda(X) = ``{\bbQ[x_1,x_2,\dotsc]^{\frS_\infty}}".
\]
Here is the precise description. For $n \in \bbZ_{>0}$, let $\Lambda^{(n)}(X) \ceq \bbQ[x_1,\dotsc,x_n]^{\frS_n}$ be the commutative $\bbQ$-algebra of $n$-variable symmetric polynomials, where each $\sigma \frS_n$ acts as $x_i \mto x_{\sigma(i)}$. We denote by $\Lambda^{(n)}(X) = \bigoplus_{d \in \bbN} \Lambda^{(n)}_d(X)$ the grading structure with respect to the degree given by $\deg x_i \ceq 1$ for each $i=1,2,\dotsc$. Then we have the projective system $\{\Lambda_n^d(X) \mid n \in \bbZ_{>0}\}$ for each $d \in \bbN$ with $\Lambda^{(n+1)}_d(X) \srj \Lambda^{(n)}_d(X)$ given by $x_{n+1} \mto 0$ and other $x_i$'s preserved. The projective limit is denoted by $\Lambda_d(X) \ceq \varprojlim_{n \to \infty} \Lambda^{(n)}_d(X)$, and the graded space 
\[
 \Lambda(X) \ceq \bigoplus_{d \in \bbN} \Lambda^d(X)
\]
has a natural structure of graded commutative $\bbQ$-algebra. This is the definition of $\Lambda(X)$. An element of $\Lambda(X)$ will be called a symmetric function of variable $X$.
Hereafter we suppress the symbol $X$ and denote $\Lambda \ceq \Lambda(X)$, $\Lambda_n \ceq \Lambda(X)$ and so on if no confusion may arise. 
We also denote $\Lambda_\bbZ$ the ring of symmetric functions with coefficients in $\bbZ$.

The projection to the ring of $n$-variable symmetric polynomials is denoted by 
\begin{align}\label{eq:is:piA}
 \pi^{(n)}_A\colon \Lambda(X) \lto \Lambda^{(n)}(X), \quad  
 \pi^{(n)}_A(x_i) = \begin{cases} x_i & (i \le n) \\ 0 & (n < i) \end{cases},
\end{align}
and sometimes called the ($n$-variable) truncation. We will also use another ring homomorphism 
\begin{align}\label{eq:is:piC}
 \pi_C^{(n)}\colon \Lambda(X) \lto \bbQ[x_1^{\pm1},\dotsc,x_n^{\pm1}]^{W_n}, \quad 
 \pi_C^{(n)}(x_i) \ceq \begin{cases} 
 x_i & (i \le n) \\ x_{i-n}^{-1} & (n < i \le 2n) \\ 0 & (2n < i) \end{cases}.
\end{align}
Here $W_n = \frS_n \ltimes (\bbZ/2\bbZ)^n$ denotes the Weyl group of the root system of type $C_n$, which acts on $x_i$'s by $\sigma(x_i)=x_{\sigma(i)}$ for $\sigma \in \frS_n$ and $\ep_i(x_j)=x_j$ ($j \ne i$), $\ep_i(x_i)=x_i^{-1}$ for $\ep_i = (0,\dotsc,0,\overset{i}{1},0,\dotsc,0) \in (\bbZ/2\bbZ)^n$. 

\begin{lem}\label{lem:is:piAC}
The families of ring homomorphisms $\{\pi^{(n)}_A \mid n \in \bbN\}$ and $\{\pi^{(n)}_C \mid n \in \bbN\}$ enjoy the following property.
\begin{enumerate}
\item 
If $f \in \Lambda$ satisfies $\pi^{(n)}_A(f)=0$ for any $n \gg 0$, then $f=0$.
\item
If $f \in \Lambda$ satisfies $\pi^{(n)}_C(f)=0$ for any $n \gg 0$, then $f=0$.
\end{enumerate}
\end{lem}

\begin{proof}
(1) is a consequence of the universality of the projective limit $\Lambda$.
(2) is proved in \cite[Lemma 3.3]{OsQ}.
\end{proof}

Finally, we introduce a bivariate version of $\Lambda(X)$. Let $X=(x_1,x_2,\dotsc)$ and $Y=(y_1,y_2,\dotsc)$ be two sequences of commuting independent variables. We define
\begin{align}\label{eq:is:LamXY}
 \Lambda(X \mid Y) \ceq \Lambda(X) \otimes \Lambda(Y),
\end{align}
where $\otimes$ denotes the tensor product over $\bbQ$ of graded commutative $\bbQ$-algebras.
Thus, $\Lambda(X \mid Y)$ is a graded commutative ring with the grading structure
\[
 \Lambda(X \mid Y) = \bigoplus_{c \in \bbN} \Lambda^c(X \mid Y), \quad 
 \Lambda^c(X \mid Y) = \bigoplus_{\substack{d,e \in \bbN \\ d+e=c}} \Lambda^{d,e}(X \mid Y), 
 \quad \Lambda^{d,e}(X \mid Y) \ceq \Lambda^d(X) \otimes_{\bbQ} \Lambda^e(Y),
\]
where $\otimes_\bbQ$ denotes the ordinary tensor product of linear spaces over $\bbQ$.
Thus, given a basis $\{P_\lambda(X) \mid \lambda\colon \text{partitions}\}$ of $\Lambda(X)$ and another $\{Q_\lambda(Y) \mid \lambda\colon \text{partitions}\}$ of $\Lambda(Y)$, we have the tensor product basis $\{P_\lambda(X) \otimes Q_\mu(Y) \mid \lambda,\mu\colon \text{partitions}\}$ of $\Lambda(X \mid Y)$.

Let us also introduce the tensor product of the ring homomorphisms $\pi^{(n)}_C$ and $\pi^{(n)}_A$. 
\begin{dfn}\label{dfn:is:piI}
Let $k,n \in \bbN$ with $k \le n$, and define a ring homomorphism 
\[
 \pi^{(k,n-k)}\colon \Lambda(X \mid Y) \lto 
 \bbQ[x_1^{\pm 1},\dotsc,x_k^{\pm 1},x_{k+1},\dotsc,x_n]^{W_k \times \frS_{n-k}}
\]
by setting
\[
 \pi^{(k,n-k)}(x_i) \ceq \begin{cases} 
 x_i & (i \le k) \\ x_{i-k}^{-1} & (k < i \le 2k) \\ 0 & (2k < i) \end{cases}, \quad 
 \pi^{(k,n-k)}(y_j) \ceq \begin{cases} 
 x_{k+j} & (j \le n-k) \\ 0 & (n-k < j) \end{cases}.
\]
\end{dfn}

Using \cref{lem:is:piAC}, we can immediately show the following statement. We omit the proof.

\begin{lem}\label{lem:is:piI}
If $f \in \Lambda(X \mid Y)$ satisfies $\pi^{(k,n-k)}(f)=0$ for any $k,n \in \bbN$ satisfying $k,n-k \gg 0$, then we have $f=0$.
\end{lem}

\subsection{Schur and symplectic Schur functions}\label{ss:is:s-ss}

Here we recall some classical bases of $\Lambda$ referring to \cite[I.2, I.3]{M95} for the detail. The following bases will be used:
\[
 \{s_\lambda \mid \lambda \in \Par\}, \quad 
 \{h_\lambda \mid \lambda \in \Par\}, \quad 
 \{e_\lambda \mid \lambda \in \Par\}, \quad 
 \{p_\lambda \mid \lambda \in \Par\},
\]
the family of Schur, completely homogeneous, elementary and power-sum symmetric functions, respectively. Let us first recall the relations 
\begin{align*}
 s_{(n)}(X) = h_n(X) \ceq 
 \sum_{1 \le i_1 \le i_2 \le \dotsb \le i_n} x_{i_1} x_{i_2} \dotsb x_{i_n}, \\
 s_{(1^n)}(X) = e_n(X) \ceq 
 \sum_{1 \le i_1 < i_2 < \dotsb < i_n} x_{i_1} x_{i_2} \dotsb x_{i_n}
\end{align*}
for each $n \in \bbN$. 
These relations are extended to the Jacobi-Trudi formulas \cite[(3.4), (3.5)]{M95}:
\begin{align}\label{eq:is:JT}
 s_\lambda(X) = \det[h_{\lambda_i -i+j}]_{i,j=1}^{\ell(\lambda)} 
              = \det[e_{\lambda_i'-i+j}]_{i,j=1}^{\ell(\lambda')},
\end{align}
where $\lambda'$ denotes the transpose of $\lambda$ (see \cref{ss:0:ntn} \ref{i:ntn:trs}). 

Next, recall the definitions $h_\lambda \ceq h_{\lambda_1} h_{\lambda_2} \dotsb h_{\lambda_l}$ and $e_\lambda \ceq e_{\lambda_1} e_{\lambda_2} \dotsb e_{\lambda_l}$ for a partition $\lambda=(\lambda_1,\dotsc,\lambda_l)$. As for the power-sum, we have $p_n(X) \ceq \sum_{i \ge 1}x_i^n$ for $n \in \bbN$ and $p_\lambda \ceq p_{\lambda_1} p_{\lambda_2} \dotsb p_{\lambda_l}$ for a partition $\lambda$. Let us also recall that $s_\lambda$'s, $h_\lambda$'s and $e_\lambda$'s are actually bases of the free module $\Lambda_\bbZ$. Finally, note that $p_\lambda$'s form a basis of $\Lambda$ since it is defined over $\bbQ$.

Now let us recall the representation-theoretic fact that the characters of irreducible polynomial modules of the general linear group $\GL_n$ over $\bbC$ are given by 
\[
 \bigl\{\pi^{(n)}_A(s_\lambda) \in \Lambda^{(n)}(X) \mid 
   \lambda \in \Par, \ \ell(\lambda) \le n\bigr\}.
\]
We call $s_\lambda(x_1,\dotsc,x_n) \ceq \pi^{(n)}_A(s_\lambda)$ the Schur polynomial as usual. 

Let us also recall the Schur functions for skew diagrams \cite[I.5]{M95}. There are several equivalent definitions, and here we only show the one extending the Jacobi-Trudi formulas \eqref{eq:is:JT}: For any pair $(\lambda, \mu)$ of partitions, we have $s_{\lambda/\mu} \in \Lambda$ satisfying the equalities 
\[
 s_{\lambda/\mu} = \det[h_{\lambda_i -\mu_j -i+j}]_{i,j=1}^n 
                 = \det[e_{\lambda'_i-\mu'_j-i+j}]_{i,j=1}^m 
\]
with $n \ge \max\{\ell(\lambda),\ell(\mu)\}$ and $m \ge \max\{\ell(\lambda'),\ell(\mu')\}$.
We also used the convention $h_r \ceq 0$ for $r<0$.
We can recover the ordinary Schur functions as $s_\lambda = s_{\lambda/\emptyset}$.
The skew Schur polynomial is defined to be 
\begin{align}\label{eq:is:ssp}
 s_{\lambda/\mu}(x_1,\dotsc,x_n) \ceq \pi^{(n)}_A(s_{\lambda/\mu}) \in 
 \Lambda^{(n)} = \bbZ[x_1,\dotsc,x_n]^{\frS_n}.
\end{align}
If $\lambda$ and $\mu$ satisfy $\lambda \supset \mu$ and $\ell(\lambda) \le n$,
then the skew Schur polynomial has the following tableau-sum formula:
\begin{align}\label{eq:is:SST}
 s_{\lambda/\mu}(x_1,\dotsc,x_n) = \sum_{T \in \SpT^{0,n}(\lambda/\mu)} x^T.
\end{align}
Here $\SpT^{0,n}(\theta)$ denotes the set of all semi-standard tableaux of shape $\theta$ with entries from the totally ordered set $\{1<2<\dotsb<n\}$. See \cite[I.5, (5.12)]{M95} for the detail, and also \cref{dfn:is:SpT} for a generalization.

Finally we recall the symplectic Schur functions \cite[Definition 2.1.1]{KT}. For a partition $\lambda$, we define 
\begin{align*}
 s^C_\lambda \ceq 
 \frac{1}{2}\det[h_{\lambda_i-i+j}+h_{\lambda_i-i-j+2}]_{i,j=1}^{\ell(\lambda)} 
 \in \Lambda.
\end{align*} 
Note that the $(i,1)$-entry of the matrix is $2 h_{\lambda_i}$, and we actually have $s^C_\lambda \in \Lambda_\bbZ$. By the ring homomorphism $\pi_C^{(n)}\colon \Lambda(X) \to \bbQ[x_1^{\pm1},\dotsc,x_n^{\pm1}]^{W_n}$ in \eqref{eq:is:piC}, we obtain the symplectic Schur polynomials
\begin{align}\label{eq:is:sCp}
 s^C_\lambda(x_1,\dotsc,x_n) = \pi^{(n)}_C(s^C_\lambda).
\end{align}
More generally, we have the skew symplectic Schur polynomials $s^C_{\lambda/\mu}(x_1, \dotsc,x_n)$ for partitions $\lambda$ and $\mu$. If $\lambda \supset \mu$ and $\ell(\lambda) \le n$, then we have the tableaux formula
\begin{align}\label{eq:is:CSpT}
 s^C_{\lambda/\mu}(x_1,\dotsc,x_n) = \sum_{T \in \SpT^{n,0}(\lambda/\mu)} x^T.
\end{align}
Here $\SpT^{n,0}(\theta)$ denotes the set of all symplectic tableaux of shape $\theta$ introduced by King \cite{K}.  See \cref{dfn:is:SpT} below for an explanation. 

The symplectic Schur function $s_\lambda^C \in \Lambda$ is a infinite-variable version of the irreducible character of the symplectic group in the following sense. 
Recall the ring homomorphism $\pi_C^{(n)}\colon \Lambda(X) \to \bbQ[x_1^{\pm1},\dotsc,x_n^{\pm1}]^{W_n}$ in \eqref{eq:is:piC}. Then the characters of irreducible rational module of the symplectic group $\Sp_{2n}$ over $\bbC$ are given by 
\[
 \bigl\{\pi_C^{(n)}(s_\lambda^C) \mid \lambda \in \Par, \ \ell(\lambda) \le n\bigr\}.
\]

\subsection{Intermediate symplectic Schur functions}\label{ss:is:is}

Let us continue to use the symbols in the previous \cref{ss:is:biv}. In particular, $X=(x_1,x_2,\dotsc)$ and $Y=(y_1,y_2,\dotsc)$ denote the infinite sequences of independent variables, and $\Lambda(X \mid Y)$  denotes the ring of bivariant symmetric functions \eqref{eq:is:LamXY}. In this subsection, we introduce the family of elements of $\Lambda(X \mid Y)$ of lifting the intermediate symplectic characters. 

\begin{dfn}\label{dfn:is:is}
For each partition $\lambda$, 
we define an element $s^I_{\lambda}(X \mid Y) \in \Lambda(X \mid Y)$ by
\[
 s^I_\lambda(X \mid Y) \ceq \sum_{\substack{\mu\colon \text{partitions} \\ \mu \subset \lambda}}
 s^C_\mu(X) s_{\lambda/\mu}(Y),
\]
and call it \emph{the intermediate symplectic Schur function}.
\end{dfn}

The name of $s^I_\lambda(X|Y)$ originates in the following \cref{prp:is:isp}. To state it, we need several preparation. Let us first recall the algebraic group $\Sp_{2k,n-k}$ introduced by Proctor \cite{P}. 

\begin{dfn}[{\cite{P}}]
Let $k,n \in \bbN$ with $k \le n$, and $V$ be the $n$-dimensional complex linear space $V \ceq \bigoplus_{i=1}^k (\bbC e_i \oplus \bbC e_{\ol{i}}) \oplus \bigoplus_{j=k+1}^n \bbC e_j$. Let $\br{\ ,\ }$ be the (possibly degenerate) skew-symmetric bilinear form on $V$ defined by
\begin{align*}
 \br{e_\alpha,e_{\beta}} \ceq 
 \begin{cases}
  1 & \quad (\alpha=i,      \, \beta=\ol{i}, \, 1\le i \le k) \\
 -1 & \quad (\alpha=\ol{i}, \, \beta=i,      \, 1\le i \le k) \\
  0 & \quad (\text{otherwise})
 \end{cases}.
\end{align*}
Then the algebraic group $\Sp_{2k,n-k}$ is defined by
\begin{align}
 \Sp_{2k,n-k} \ceq \sbr{g\in\GL_{n+k} \mid \forall v,w \in V, \ \br{gv,gw}=\br{v,w}}.
\end{align}
Note that if $k=n$ or $0$, we have the following isomorphisms of algebraic groups, respectively.
\begin{align*}
 \Sp_{2n,0} \cong \Sp_{2n}, \quad \Sp_{0,n} \cong \GL_n.
\end{align*}
We call it \emph{the intermediate symplectic group}.
\end{dfn} 

By \cite{P}, finite-dimensional indecomposable weight module of $\Sp_{2k,n-k}$ are parametrized by $\Par$. We denote the character of the indecomposable module $V_\lambda$ corresponding to $\lambda \in \Par$ by 
\begin{align}\label{eq:is:spf}
 \spf^{(k,n-k)}_\lambda(x_1,\dotsc,x_n) \in 
 \bbZ[x_1^{\pm1},\dotsc,x_k^{\pm1},x_{k+1},\dotsc,x_n]^{W_k \times \frS_{n-k}}.
\end{align}
and call it the intermediate symplectic character. It has a tableau-sum formula explained in \cref{fct:is:ISpT} below.

\begin{dfn}[{\cite{P}, \cite[Definition 2.1]{Oisp}}]\label{dfn:is:SpT}
Let $k,n \in \bbN$ with $k \le n$, and $\theta$ be a skew diagram with $\ell(\theta) \le n$. A $(k,n-k)$-symplectic tableau of shape $\theta$ is a filling of the cells of $\theta$ with entries from the totally ordered set 
\[
 \{1<\ol{1}<2<\ol{2}<\cdots<k<\ol{k}<k+1<\dots<n\}
\]
satisfying the following rules.
\begin{enumerate}[label=(ST\arabic*)]
 \item the entries in each row weakly increasing from left to right;
 \item the entries in each column strictly increasing from top to bottom.
 \item the entries in the $i$-th row are greater than or equal to $i$ for each $i=1,\dotsc,n$.
\end{enumerate}
We denote by $\SpT^{(k,n-k)}(\theta)$ 
the set of all $(k,n-k)$-symplectic tableaux of shape $\theta$.
\end{dfn}

Note that $\SpT^{(0,n)}(\theta)$ is equal to the set of semi-standard tableaux (c.f.\ \eqref{eq:is:SST}), and $\SpT^{(n,0)}(\theta)$ is equal to the set of symplectic tableaux (c.f.\ \eqref{eq:is:CSpT}).

\begin{fct}[{\cite{P}}]\label{fct:is:ISpT}
Let $k,n \in \bbN$ with $\lambda$ be a partition with $\ell(\lambda) \le n$. Then the intermediate symplectic character \eqref{eq:is:spf} has the presentation 
\[
 \spf^{(k,n-k)}_{\lambda}(x_1,\dotsc,x_n) = \sum_{T \in \SpT^{(k,n-k)}(\lambda)} x^T
\]
with $x^T \ceq \prod_{i=1}^k x_i^{\#\{\text{$i$'s in $T$}\}-\#\{\text{$\ol{i}$'s in $T$}\}}\prod_{i=k+1}^n x_i^{\#\{\text{$i$'s in $T$}\}}$.
\end{fct}

In the case $k=n$ or $k=0$, this \cref{fct:is:ISpT} recovers the tableau-sum formula \eqref{eq:is:SST} of the Schur polynomial
\[
 s_{\lambda}(x_1,\dotsc,x_n) \ceq \pi_A^{(n)}(s_{\lambda}) 
 \in \bbZ[x_1,\dotsc,x_n]^{\frS_n}
\]
and that \eqref{eq:is:CSpT} of the symplectic Schur polynomial 
\[
 s^C_{\lambda}(x_1,\dotsc,x_n) \ceq \pi_C^{(n)}(s^C_{\lambda}) 
 \in \bbZ[x_1^{\pm1},\dotsc,x_n^{\pm1}]^{W_n}.
\]

Now we can explain the origin of the name of $s^I_{\lambda/\mu}(X \mid Y)$. 

\begin{prp}\label{prp:is:isp}
Let $\lambda$ be a partition. The symmetric function $s^I_\lambda(X \mid Y) \in \Lambda(X \mid Y)$ is uniquely characterized by the following property: For $k,n \in \bbN$ with $k \le n$ and $\ell(\lambda) \le n$, we have 
\[
 \pi^{(k,n-k)}\bigl(s^I_\lambda(X \mid Y)\bigr) = \spf^{(k,n-k)}_\lambda(x_1,\dotsc,x_n),
\]
where $\pi^{(k,n-k)}$ is the ring homomorphism in \cref{dfn:is:piI}.
\end{prp}

\begin{proof}
The uniqueness follows from \cref{lem:is:piI}. Let us show that $s^I_\lambda(X \mid Y)$ satisfies the equality. By \cref{dfn:is:is} of $s^I_\lambda$ and \cref{dfn:is:piI} of $\pi^{(k,n-k)}$, the left hand side is equal to
\[
 \sum_{\mu \subset \lambda} 
 \pi^{(k)}_C\bigl(s^C_\mu(X)\bigr) \pi^{(n-k)}_A\bigl(s_{\lambda/\mu}(Y)\bigr).
\]
By \eqref{eq:is:ssp} and \eqref{eq:is:sCp}, it can be rewritten as 
\[
 \sum_{\mu \subset \lambda} 
 s^C_\mu(x_1,\dotsc,x_k) s_{\lambda/\mu}(x_{k+1},\dotsc,x_n),
\]
which is by the tableau-sum formula \eqref{eq:is:SST} and \eqref{eq:is:CSpT} equal to 
\begin{align}\label{eq:is:TCTA}
 \sum_{\mu \subset \lambda} 
 \sum_{T_C \in \SpT^{k,0}(\mu)} x^{T_C} \sum_{T_A \in \SpT^{0,n-k}([k+1,n];\lambda/\mu)} x^{T_A},
\end{align}
where $\SpT^{k,0}(\mu)$ denotes the set of symplectic tableaux of shape $\mu$ with entries from the totally ordered set $\{1<\ol{1}<2<\ol{2}<\dotsb<k<\ol{k}\}$, and $\SpT^{0,n-k}([k+1,n];\lambda/\mu)$ denotes the set of semi-standard tableaux of shape $\lambda-\mu$ with entries from the totally ordered set $\{k+1<k+2<\dotsb<n\}$. By \cref{dfn:is:SpT}, the product set $\SpT^{k,0}(\mu) \times \SpT^{0,n-k}([k+1,n];\lambda/\mu)$ is equal to the subset of $\SpT^{(k,n-k)}(\lambda)$ consisting of tableaux whose entries $1,\ol{1},\dotsc,k,\ol{k}$ occupy the shape $\mu$. Thus, the summation \eqref{eq:is:TCTA} is equal to
\[
 \sum_{T \in \SpT^{(k,n-k)}(\lambda)} x^T,
\]
which is equal to $\spf^{(k,n-k)}_\lambda(x_1,\dotsc,x_n)$ by \cref{fct:is:ISpT}. 
\end{proof}

Now recall that we have a natural embedding of graded rings 
\begin{align}\label{eq:is:iota}
 \iota_{X,Y}\colon \Lambda(X \cup Y) \linj \Lambda(X \mid Y) = \Lambda(X) \otimes \Lambda(Y), 
\end{align}
where $\Lambda(X \cup Y)$ denotes the ring of symmetric functions with variables $X \cup Y$.
The following  we find that the symmetric function $s^I_\lambda(X \mid Y)$ actually lives in the smaller space $\Lambda(X \cup Y)$.

\begin{prp}\label{prp:is:sI=sC}
For any partition $\lambda$, we have the equality
\[
 s^I_\lambda(X \mid Y) = s^C_\lambda(X \cup Y),
\]
where the right hand side denotes the symplectic Schur function of variables $X \cup Y$,
living in $\Lambda(X \cup Y)$.
\end{prp}

\begin{proof}
By \cite[Corollary 2.6]{Oisp}, for any partition $\lambda$ satisfying $\ell(\lambda) \le k+1$, 
we have
\[
 \spf^{(k,n-k)}_\lambda(x_1,\dotsc,x_n) = 
 s^C_\lambda(x_1,x_1^{-1},\dotsc,x_k,x_k^{-1},x_{k+1},\dotsc,x_n,0,0,\dotsc).
\]
The left hand side is $\pi^{(k,n-k)}\bigl(s^I_\lambda(X\mid Y)\bigr)$, and 
the right hand side is $(\pi^{(k,n-k)}\circ\iota_{X,Y})\bigl(s^C_\lambda(X \cup Y)\bigr)$.
Then \cref{lem:is:piI} on $\pi^{(k,n-k)}$ yields $s^I_\lambda(X \mid Y) = s^C_\lambda(X \cup Y)$.
\end{proof}

\section{Intermediate symplectic $Q$-functions}\label{s:iQ}

\subsection{Schur's $Q$-functions and symplectic $Q$-functions}

Here we give a summary on Schur's $Q$-function \cite[III.8]{M95} and its symplectic analogue, the symplectic $Q$-function \cite[\S3]{OsQ}.

In the ring $\Lambda=\Lambda(X)$ of symmetric functions with variables $X=(x_1,x_2,\dotsc)$, we define $q^A_r \in \Lambda$ ($r \in \bbN$) by the generating series as
\[
 \sum_{r \ge 0} q^A_r z^r = \prod_{i \ge 1} \frac{1+x_i z}{1-x_i z}.
\]
We denote by $\Gamma \subset \Lambda$ the graded subalgebra generated by $q_r$'s, and set $\Gamma_d \ceq \Lambda_d \cap \Gamma$. 

Now, we define $Q^A_\lambda \in \Lambda$ for each \emph{strict} partition $\lambda$ (see \cref{ss:0:ntn}, \ref{i:ntn:strpar}) inductively on the length $\ell(\lambda)$. In the case of length $0,1,2$, we set
\begin{align}\label{eq:iQ:QArec}
 Q^A_{\emptyset} \ceq 1, \quad Q^A_{(r)} \ceq q^A_r \quad (r>0), \quad 
 Q^A_{(r,s)} \ceq q^A_r q^A_s+2\sum_{k=1}^s (-1)^k q_{r+k} q_{r-k} \quad (r>s>0).
\end{align}
Then, in the case of length $\ge 3$, we define
\begin{align}\label{eq:iQ:QA-M}
 Q^A_\lambda \ceq \Pf\bigl[Q^A_{(\lambda_i,\lambda_j)}\bigr]_{1 \le i,j \le m}.
\end{align}
Here $\Pf$ denotes the Pfaffian of an even-size skew-symmetric matrix, and $m \ceq \ell(\lambda)$ or $\ell(\lambda)+1$ according whether $\ell(\lambda)$ is even or odd. In the case $\ell(\lambda)$ is odd, we put $\lambda_m \ceq 0$. Finally, for the entries of the matrix, we used the convention $Q^A_{(r,0)} \ceq Q^A_{(r)}$, $Q^A_{(s,r)} \ceq Q^A_{(r,s)}$ for $s>r$, and $Q^A_{(r,r)} \ceq 0$ for $r \in \bbN$. Then the family $\{Q^A_\lambda \mid \lambda\colon \text{partitions}\}$ is a basis of $\Gamma$. We call $Q^A_\lambda(X)$ \emph{Schur's $Q$-function}.

Let us also recall that Schur's $Q$-function can be extended to the skew diagrams. Let $\lambda=(\lambda_1,\dotsc,\lambda_l)$ and $\mu=(\mu_1,\dotsc,\mu_m)$ be strict partitions such that $\lambda_l>0$ and $\mu_m \ge 0$. We may assume that $l+m$ is even since we can replace $m$ by $m+1$ if $l+m$ is odd. Then we define
\begin{align}\label{eq:iQ:QAlm}
 Q^A_{\lambda/\mu} \ceq 
 \Pf\begin{bmatrix} M^A_\lambda & N^A_{\lambda,\mu} \\ -\trs{N^A_{\lambda,\mu}} & O \end{bmatrix}, 
\end{align}
where $M^A_\lambda \ceq \bigl[Q^A_{(\lambda_i,\lambda_j)}\bigr]_{1 \le i,j \le l}$ as in \eqref{eq:iQ:QA-M}, and $N^A_{\lambda,\mu} \ceq \bigl[q^A_{\lambda_i-\mu_{m+j-1}}\bigr]_{1 \le i \le l, 1 \le j \le m}$. We call $Q^A_{\lambda/\mu}$ \emph{skew Schur's $Q$-function}. We can recover \eqref{eq:iQ:QA-M} by $Q^A_{\lambda/\emptyset}=Q^A_\lambda$. The Pfaffian formula \eqref{eq:iQ:QAlm} is originally due to J\'ozefiak and Pragacz \cite{JP}. See \cite[III.8, Exercise 9]{M95} for a brief account, and \cite[\S 6, Theorem 6.2]{S} for a combinatorial proof based on the Lindstr\"{o}m-Gessel-Viennot theorem \cite{GV,L}.

For strict partitions $\lambda,\mu$ with $\ell(\lambda) \le n$, the truncation $\pi^{(n)}_A\colon \Lambda \to \Lambda^{(n)}$ in \eqref{eq:is:piA} yields \emph{skew Schur's $Q$-polynomial} 
\begin{align}\label{eq:iQ:QAp}
 Q^A_{\lambda/\mu}(x_1,\dotsc,x_n) \ceq \pi^{(n)}_A(Q^A_{\lambda/\mu}).
\end{align}
We denote $Q^A_{\lambda}(x_1,\dotsc,x_n) \ceq Q^A_{\lambda/\emptyset}(x_1,\dotsc,x_n)$ and call it \emph{Schur's $Q$-polynomial}. The polynomial $Q^A_{\lambda/\mu}$ vanishes unless $\lambda \supset \mu$.

For strict partitions $\lambda$ and $\mu$ satisfying $\ell(\lambda) \le n$ and $\lambda \supset \mu$, Schur's $Q$-polynomial has the tableau-sum formula \cite[III.8, (8.16')]{M95}:
\begin{align}\label{eq:iQ:QA-T}
 Q^A_{\lambda/\mu}(x_1,\dotsc,x_n) = \sum_{T \in \QT^{(0,n)}(\lambda/\mu)} x^T,
\end{align}
where $\QT^{(0,n)}(\lambda/\mu)$ denotes the set of marked shifted tableaux of shape $S(\lambda/\mu)$  in the sense of \cite[III.8]{M95}. We refer to \cref{rmk:iQ:QT} for the definition of a marked shifted tableau. As for the symbol $S(\lambda/\mu)$, we have:

\begin{dfn}
Let $\lambda$ and $\mu$ be partitions.
\begin{enumerate}
\item 
The \emph{shifted diagram} $S(\lambda)$ of $\lambda$ is defined to be
\[
 S(\lambda) \ceq
 \{(i,j) \in \bbZ^2 \mid 1 \le i \le \ell(\lambda), \ i \le j \le \lambda_i+i-1\},
\]
which will be depicted in the same way as the ordinary Young diagrams 
(see \cref{ss:0:ntn}, \ref{i:ntn:Y}).
Note that $S(\lambda)$ is a partition if and only if $\lambda$ is strict (\cref{ss:0:ntn}, \ref{i:ntn:strpar}).

\item
If $\lambda \supset \mu$, then we define the shifted skew diagram $S(\lambda/\mu)$ as the set-theoretical difference $S(\lambda/\mu) \ceq S(\lambda)-S(\mu)$. 
\end{enumerate}
\end{dfn}

Next, we recall the symplectic $Q$-functions \cite[\S3]{OsQ}. 
First we define $q^C_r \in \Lambda$ ($r \in \bbN$) by 
\[
 \sum_{r \ge 0} q^C_r z^r = 
 \prod_{i \ge 1} \frac{1+x_i z}{1-x_i z}\frac{1+x_i^{-1} z}{1-x_i^{-1} z}.
\]
Then we define $Q^A_\lambda \in \Lambda$ for each strict partition $\lambda$ inductively on the length $\ell(\lambda)$ as
\begin{gather}
\nonumber
 Q^C_{\emptyset} \ceq 1, \quad Q^C_{(r)} \ceq q^C_r \quad (r>0), \\
\nonumber
 Q^C_{(r,s)} \ceq q^C_r q^C_s+2\sum_{k=1}^s (-1)^k \bigl(q^C_{r+k}
 +2\sum_{i=1}^{k-1}q^C_{r+k-2i}+q^C_{r-k}\bigr) q^C_{s-k} \quad (r>s>0), \\ 
\label{eq:iQ:QC-M}
 Q^C_\lambda \ceq \Pf\bigl(Q_{(\lambda_i,\lambda_j)}\bigr)_{1 \le i,j \le m}.
\end{gather}
The size $m$ of the matrix in \eqref{eq:iQ:QC-M} is given by $\ell(\lambda)$ if it is even, and by $\ell(\lambda)+1$ otherwise. Except the summation term for $Q^C_{(r,s)}$, the recursion is almost the same as \eqref{eq:iQ:QArec}. We call the obtained family $\{Q^C_\lambda \in \Lambda \mid \lambda \in \SPar \}$ \emph{the symplectic $Q$-functions}.

We also have \emph{the skew symplectic $Q$-function} \cite[(3.14)]{OsQ}. For strict partitions $\lambda=(\lambda_1,\dotsc,\lambda_l)$ and $\mu=(\mu_1,\dotsc,\mu_m)$ such that $\lambda_l>0$, $\mu_m \ge 0$ and $l+m$ is even, we define
\begin{align}\label{eq:iQ:QClm}
 Q^C_{\lambda/\mu} \ceq 
 \Pf\begin{bmatrix} M^C_\lambda & N^C_{\lambda,\mu} \\ -\trs{N^C_{\lambda,\mu}} & O \end{bmatrix}, 
\end{align}
where $M^C_\lambda \ceq \bigl[Q^C_{(\lambda_i,\lambda_j)}\bigr]_{1 \le i,j \le l}$ and $N^C_{\lambda,\mu} \ceq \bigl[q^C_{\lambda_i-\mu_{m+1-j}}\bigr]_{1 \le i \le l, 1 \le j \le m}$.
We can recover \eqref{eq:iQ:QC-M} by $Q^C_{\lambda/\emptyset}=Q^C_\lambda$.

Let $\pi^{(n)}_C\colon \Lambda \to \bbQ[x_1^{\pm1},\dotsc,x_n^{\pm 1}]^{W_n}$ be the ring homomorphism in \eqref{eq:is:piC}. For strict partitions $\lambda$ and $\mu$ with $\ell(\lambda) \le n$, we define
\[
 Q^C_{\lambda/\mu}(x_1,\dotsc,x_n) \ceq \pi^{(n)}_C(Q^C_{\lambda/\mu})
\]
and call it \emph{the skew symplectic $Q$-polynomial}. 
The case $Q^C_\lambda(x_1,\dotsc,x_n) \ceq Q^C_{\lambda/\emptyset}(x_1,\dotsc,x_n)$ is called \emph{the symplectic $Q$-polynomial}. The polynomial $Q^C_{\lambda/\mu}(x_1,\dotsc,x_n)$ vanishes unless $\lambda \supset \mu$, and in the case $\lambda \supset \mu$ it has the tableau-sum formula \cite[Theorem 4.2]{OsQ}: 
\begin{align}\label{eq:iQ:QC-T}
 Q^C_{\lambda/\mu}(x_1,\dotsc,x_n) = \sum_{T \in \QT^{(k,0)}(\lambda/\mu)} x^T,
\end{align}
where $\QT^{(n,0)}(\lambda/\mu)$ denotes the set of symplectic marked shifted tableaux of shape $S(\lambda/\mu)$. See \cref{rmk:iQ:QT} for the detail.

\begin{rmk}
In \cite{OsQ}, the symmetric function $Q^C_{\lambda/\mu}(X)$ is called the universal symplectic $Q$-function, and the Laurent polynomial $Q^C_{\lambda/\mu}(x_1,\dotsc,x_n)$ is called the symplectic $Q$-function. Our terminology follows the principle \ref{i:ntn:sf} in \cref{ss:0:ntn}.
\end{rmk}

\subsection{Intermediate symplectic $Q$-function}

Now we introduce the intermediate analogue of Schur's $Q$- and symplectic $Q$-functions. 

\begin{dfn}\label{dfn:iQ:iQf}
Let $X$ and $Y$ be two infinite sequences of variables.
For strict partitions $\lambda$ and $\mu$ such that $\lambda \supset \mu$, we set
\[
 Q^{I}_{\lambda/\mu}(X \mid Y) \ceq \sum_{\nu} Q^C_{\lambda/\nu}(X) Q^A_{\nu/\mu}(Y),
\]
where the sum is taken over the strict partitions $\nu$ 
such that $\lambda \supset \nu$ and $\nu \supset \mu$.
We call it \emph{the intermediate symplectic $Q$-function}.
\end{dfn}

Let us also introduce the polynomial version.
Let $k,n \in \bbN$ satisfy $k \le n$. Recall the ring homomorphism in \cref{dfn:is:piI}:
\begin{align*}
&\pi^{(k,n-k)}\colon \Lambda(X \mid Y) \lto \Lambda^{(k,n-k)}(X) \ceq 
 \bbQ[x_1^{\pm 1},\dotsc,x_k^{\pm 1},x_{k+1},\dotsc,x_n]^{W_k \times \frS_{n-k}}, \\
&\pi^{(k,n-k)}(x_i) \ceq \begin{cases} 
 x_i & (i \le k) \\ x_{i-k}^{-1} & (k < i \le 2k) \\ 0 & (2k < i) \end{cases}, \quad 
 \pi^{(k,n-k)}(y_j) \ceq \begin{cases} 
 x_{k+j} & (j \le n-k) \\ 0 & (n-k < j) \end{cases}.
\end{align*}
We sometimes denote $x = (x_1,\dotsc,x_k \mid x_{k+1},\dotsc,x_n)$ to distinguish 
the former and latter parts. 

\begin{dfn}\label{dfn:iQ:iQp}
Let $k$ and $n$ be as above.
For strict partitions $\lambda$ and $\mu$ with $\ell(\lambda) \le n$, we define
\[
 Q^{(k,n-k)}_{\lambda/\mu}(x_1,\dotsc,x_n) \ceq \pi^{(k,n-k)}(Q^I_{\lambda/\mu})
\]
and call it \emph{the intermediate symplectic $Q$-polynomial}.
\end{dfn}

\begin{rmk}\label{rmk:iQ:QAC}
The definition immediately gives
\begin{align*}
 Q^{n,0}_{\lambda/\mu}(x_1,\dotsc,x_n) = \pi^{(n)}_C(Q^C_{\lambda/\mu}) = 
 Q^C_{\lambda/\mu}(x_1,\dotsc,x_n), \\
 Q^{0,n}_{\lambda/\mu}(x_1,\dotsc,x_n) = \pi^{(n)}_A(Q^A_{\lambda/\mu}) = 
 Q^A_{\lambda/\mu}(x_1,\dotsc,x_n),
\end{align*}
i.e., the specialization recovers Schur's $Q$-polynomial and the symplectic $Q$-polynomial, respectively. This is the origin of the name ``intermediate'' symplectic $Q$-polynomial.
\end{rmk}

As the $Q$-polynomials $Q^A_{\lambda/\mu}(x)$ and $Q^C_{\lambda/\mu}(x)$ enjoy the tableau-sum formula \eqref{eq:iQ:QA-T} and \eqref{eq:iQ:QC-T}, the intermediate polynomials $Q^{k.n-k}_{\lambda/\mu}(x)$ also has a tableau-sum formula, stated in \cref{prp:iQ:iQ-T} below. 
As a preparation, let us introduce:

\begin{dfn}\label{dfn:iQ:QT}
Let $k,n \in \bbN$ satisfy $k \le n$, 
and $\lambda,\mu$ be strict partitions such that $\lambda \supset \mu$. 
An \emph{intermediate symplectic primed shifted tableau} of shape $S(\lambda/\mu)$ is a filling of the cells of $S(\lambda/\mu)$ with entries from the totally ordered set
\begin{align}\label{eq:iQ:k-n}
\begin{split}
 \{1'&<1<\ol{1}'<\ol{1}<2'<2<\ol{2}'<\ol{2}<\dotsb<k'<k<\ol{k}'<\ol{k} \\
     &<(k+1)'<k+1<(k+2)'<k+2<\dotsb<n'<n\}
\end{split}
\end{align}
satisfying the following conditions.
\begin{enumerate}[label=(QT\arabic*)]
\item the entries in each row weakly increasing from left to right;
\item the entries in each column weakly increasing from top to bottom;
\item each row contains at most one $i'$ and at most one $\ol{j'}$ for 
      each $i \in \{1,\dotsc,n\}$ and $j \in \{1,\dotsc,k\}$;
\item each column contains at most one $i$ and at most one $\ol{j}$ for 
      each $i \in \{1,\dotsc,n\}$ and $j \in \{1,\dotsc,k\}$;
\item the entry of the $i$-th entry on the main diagonal is one of $\{i',i,\ol{i}',\ol{i}\}$
      for each $i \in \{1,\dotsc,k\}$.
\end{enumerate}
We denote by $\QT^{(k,n-k)}(\lambda/\mu)$ the set of all the intermediate symplectic primed shifted tableaux of shape $S(\lambda/\mu)$. We also denote $\QT^{(k,n-k)}(\lambda) \ceq \QT^{(k,n-k)}(\lambda/\emptyset)$.
\end{dfn}

\begin{rmk}\label{rmk:iQ:QT}
Our definition is an intermediate version of the following 
``type $A$'' and ``type $C$'' notions of primed shifted tableaux.
\begin{itemize}
\item
In the case $k=0$, the set $\QT^{(0,n)}(\theta)$ coincides with the set $\QT^A(\theta)$ of marked shifted tableaux of shape $S(\theta)$ in the sense of \cite[III.8]{M95}. 
The corresponding conditions are given in \cite[p.256, (M1)--(M3)]{M95}.
\item
In the case $n-k=0$, the set $\QT^{(n,0)}(\theta)$ coincides with the set $\QT^C(\theta)$ of symplectic primed shifted tableaux of shape $S(\theta)$ in the sense of \cite[\S2.2]{HK}. 
The corresponding conditions are given in \cite[{PST1--PST4, QST$\ol{5}$}]{HK}.
\end{itemize}
\end{rmk}

\begin{prp}\label{prp:iQ:iQ-T}
Let $k,n \in \bbN$ satisfy $k \le n$, and $x=(x_1,\dotsc,x_n)$ be a sequence of variables.
Also let $\lambda$ and $\mu$ be strict partitions such that $\lambda \supset \mu$.
For $T \in \QT^{(k,n-k)}(\lambda/\mu)$, we set 
\[
 x^T \ceq \prod_{i=1}^k x_i^{m(i')+m(i)-m(\ol{i}')-m(\ol{i})} 
          \prod_{j=k+1}^n  x_j^{m(j')+m(j)},
\]
where $m(\gamma)$ denotes the multiplicity of the entry $\gamma$ in $T$. 
If $\ell(\lambda) \le n$, then we have 
\begin{align}\label{eq:iQ:Qp}
 Q^{(k,n-k)}_{\lambda/\mu}(x_1,\dotsc,x_n) = \sum_{T \in \QT^{(k,n-k)}(\lambda/\mu)}x^T.
\end{align}
\end{prp}

\begin{proof}
As in the proof of \cref{prp:is:isp}, 
the left hand side of \eqref{eq:iQ:Qp} is equal to
\[
 \sum_{\nu} Q^C_{\nu/\mu}(x_1,\dotsc,x_k) Q^A_{\lambda/\mu}(x_{k+1},\dotsc,x_n),
\]
where the sum is taken over the strict partitions $\nu$ such that $\mu \subset \nu \subset \lambda$. By the tableau-sum formulas \eqref{eq:iQ:QA-T} and \eqref{eq:iQ:QC-T}, it is equal to 
\begin{align}\label{eq:iQ:TCTA}
 \sum_{\nu} 
 \sum_{T_C \in \QT^{(k,0)}(\nu/\mu)} x^{T_C} \sum_{T_A \in \QT^{(0,n-k)}([k+1,n];\lambda/\nu)} x^{T_A},
\end{align}
where 
$\QT^{(0,n-k)}([k+1,n];\lambda/\nu)$ denotes the set of marked shifted tableaux of shape $S(\lambda/\nu)$ with entries from the totally ordered set $\{k+1<(k+1)'<\dotsb<n'<n\}$. By \cref{dfn:iQ:QT} and \cref{rmk:iQ:QT}, the product set $\QT^{(k,0)}(\nu/\mu) \times \QT^{(0,n-k)}([k+1,n];\lambda/\nu)$ is equal to the subset of $\QT^{(k,n-k)}(\lambda/\mu)$ consisting of tableaux whose entries $1',1,\ol{1}',\ol{1},\dotsc,k',k,\ol{k}',\ol{k}$ occupy the shifted skew diagram $S(\nu/\mu)$. Thus, the summation \eqref{eq:iQ:TCTA} is equal to the right hand side of \eqref{eq:iQ:Qp}.
\end{proof}

%
%

\subsection{Formulas of intermediate symplectic $Q$-polynomials}

Here we study basic properties of the (Laurent) polynomial $Q^{k:n-k}_{\lambda/\mu}(x_1,\dotsc,x_n)$. We may take \eqref{eq:iQ:Qp} for its definition: 
\[
 Q^{(k,n-k)}_{\lambda/\mu}(x_1,\dotsc,x_n) = \sum_{T \in \QT^{(k,n-k)}(\lambda/\mu)}x^T.
\]

\begin{lem}\label{lem:iQ:tab}
Let $k,n \in \bbN$ and $x=(x_1,\dotsc,x_n)$ be as in \cref{prp:iQ:iQ-T}.
Also, let $\lambda$ and $\mu$ be strict partitions with $\lambda \supset \mu$.
\begin{enumerate}
\item 
For an additional indeterminate $z$,
we have the following equality of formal series.
\begin{align*}
 \sum_{l \in \bbN} Q^{(k,n-k)}_{(l)}(x_1,\dotsc,x_n) z^l = 
 \prod_{i=  1}^k \frac{1+x_i z}{1-x_i z} \frac{1+x_i^{-1}z}{1-x_i^{-1}z}
 \prod_{j=k+1}^n \frac{1+x_j z}{1-x_j z}.
\end{align*}
\item
We have
\begin{align}\label{eq:iQ:l=1}
 Q^{(k,n-k)}_{\lambda/\mu}(x_1,\dotsc,x_n) = 
 \sum \prod_{i=1}^k   Q^C_{\mu^{(i)}/\mu^{(i-1)}}(x_i) 
      \prod_{j=k+1}^n Q^A_{\mu^{(j)}/\mu^{(j-1)}}(x_j),
\end{align}
where the sum is taken over all the sequences 
$\mu=\mu^{(0)} \subset \mu^{(1)} \subset \dotsb \subset \mu^{(n-1)} \subset \mu^{(n)}=\lambda$.
\item \label{i:lem:iQ:tab:3}
For each single variable $x_i$ with $i=1,\dotsc,k$, we have 
\[
 Q^C_{\lambda/\mu}(x_i) = \begin{cases} 0 & (\ell(\lambda)-\ell(\mu) > 1) \\ 
 \det\bigl[Q^C_{(\lambda_l-\mu_m)}(x_i)\bigr]_{l,m=1}^{\ell(\lambda)} & 
 (\ell(\lambda)-\ell(\mu) \le 1) \end{cases}
\]
with the convention $Q^C_{(r)}(x)=0$ for $r<0$. Similarly, for $j=k+1,\dotsc,n$, we have
\[
 Q^A_{\lambda/\mu}(x_j) = \begin{cases} 0 & (\ell(\lambda)-\ell(\mu) > 1) \\ 
 \det\bigl[Q^A_{(\lambda_l-\mu_m)}(x_j)\bigr]_{l,m=1}^{\ell(\lambda)} & 
 (\ell(\lambda)-\ell(\mu) \le 1) \end{cases}.
\]
\end{enumerate}
\end{lem}

\begin{proof}
\begin{enumerate}
\item 
We denote the left hand side by $q(z)$, and take $l \in \bbN$ and $T \in \QT^{(k,n-k)}\bigl((l)\bigr)$. Then for each $i=1,\dotsc,n$, the letter $i'$ can appear at most once in $T$ by the condition (QT3) in \cref{dfn:iQ:QT}, 
and the letter $i$ can appear in arbitrary times. Thus the variable $x_i$ contributes 
$(1+x_i z)(1+x_i z+x_i^2z^2+\dotsb)=(1+x_i z)/(1-x_i z)$ to the generating series $q(z)$.
Similarly, for each $j \in \{1,\dotsc,k\}$, 
the letter $\ol{j}'$ can appear at most once in $T$ by the condition (QT4) in \cref{dfn:iQ:QT}, 
and the letter $\ol{j}$ can appear in arbitrary times. Thus the variable $x_j$ contributes
$(1+x_j^{-1} z)(1+x_j^{-1} z+x_j^{-2}z^2+\dotsb)=(1+x_j^{-1} z)/(1-x_j^{-1} z)$ to $q(z)$.
Hence we have the claim.
\item
We obtain the formula by decomposing a tableau $T \in \QT^{(k,n-k)}(\lambda/\mu)$ in the definition \eqref{eq:iQ:Qp} into the subtableaux consisting of $i',i,\ol{i}',\ol{i}$ for $i=1,\dotsc,k$, and of $j',j$ for $j=k+1,\dotsc,n$.
\item
The first half of the statement is shown in \cite[Lemma 4.4, (2), (3)]{OsQ}. The second half can also be shown by a similar argument in loc.\ cit. 
\end{enumerate}
\end{proof}

Below is an intermediate analogue of the J\'ozefiak-Pragacz type formulas \eqref{eq:iQ:QAlm} and \eqref{eq:iQ:QClm}.

\begin{thm}\label{thm:iQ:iJP}
Let $k,n \in \bbN$ satisfy $k \le n$, and $\lambda=(\lambda_1,\dotsc,\lambda_l)$ and $\mu=(\mu_1,\dotsc,\mu_m)$ be strict partitions satisfying $\lambda \supset \mu$, $\lambda_l>0$, $\mu \ge 0$, $l \ge 2$ and $l+m$ even. Then we have
\[
 Q^{(k,n-k)}_{\lambda/\mu}(x_1,\dotsc,x_n) = \Pf
 \begin{bmatrix} M^I_\lambda & N^I_{\lambda,\mu} \\ -\trs{N^I_{\lambda,\mu}} & O \end{bmatrix},
\]
where $M^I_\lambda \ceq \bigl[ Q^{(k,n-k)}_{(\lambda_i,\lambda_j)}(x_1,\dotsc,x_n) \bigr]_{i,j=1}^l$ and $N^I_{\lambda,\mu} \ceq \bigl[Q^{(k,n-k)}_{(\lambda_i-\mu_{m+1-j})}(x_1,\dotsc,x_n)\bigr]_{1 \le i \le l, 1 \le j \le m}$.
\end{thm}

\begin{proof}
We use an ``intermediate'' modification of Stembridge's variation \cite[\S6]{S} of the Lindstr\"{o}m-Gessel-Viennot theorem \cite{GV,L}, and make a similar argument as \cite[Theorem 6.2]{S}. Let us consider the directed graph $\Gamma^{(k,n-k)}$ shown in the left of \cref{fig:iQ:LGV}.  The vertex set $V$ is given by $V= U \cup I$ with 
\[
 U \ceq \{(x,y) \in \bbN^2 \mid x \ge 1, \ y \le k+n\}, \quad 
 I \ceq \{(0,y) \mid y \in \thf\bbN, \ y \le k+n\}
\]
and the edge set $E$ is given by $E \ceq H \cup \ol{H} \cup H' \cup P \cup D \cup \ol{D} \cup D' \cup D_0 \cup \ol{D}_0 \cup D_0'$ with 
\begin{align*}
     H &\ceq \{(i-1,j)   \to (i,j) \mid (i-1,j),(i,j) \in V, \ j \in 2\bbN+1, \ j \le 2k \}, \\
 \ol{H}&\ceq \{(i-1,j)   \to (i,j) \mid (i-1,j),(i,j) \in V, \ j \in 2\bbN,   \ j \le 2k \}, \\
     H'&\ceq \{(i-1,j)   \to (i,j) \mid (i-1,j),(i,j) \in V, \ j > 2k\}, \\
     P &\ceq \{(i,j-1)   \to (i,j) \mid (i,j-1),(i,j) \in V\}, \\
     D &\ceq \{(i-1,j-1) \to (i,j) \mid (i-1,j-1),(i,j) \in U, \ j \in 2\bbN+1, \ j \le 2k\}, \\
 \ol{D}&\ceq \{(i-1,j-1) \to (i,j) \mid (i-1,j-1),(i,j) \in U, \ j \in 2\bbN,   \ j \le 2k\}, \\
     D'&\ceq \{(i-1,j-1) \to (i,j) \mid (i-1,j-1),(i,j) \in U, \ j > 2k\}, \\
   D_0 &\ceq \{(0,j-\thf)\to (1,j) \mid (0,j-\thf),(1,j) \in V, \ j \in 2\bbN+1, \ j \le 2k\}, \\
\ol{D}_0&\ceq\{(0,j-\thf)\to (1,j) \mid (0,j-\thf),(1,j) \in V, \ j \in 2\bbN,   \ j \le 2k\}, \\
   D_0'&\ceq \{(0,j-\thf)\to (1,j) \mid (0,j-\thf),(1,j) \in V, \ j > 2k\}.
\end{align*}
Note that we distinguish the regions $j \le 2k$ and $j>2k$.

We denote by $\clP(s,t)$ for $s,t \in V$ the set of path from $s$ to $t$ in the directed graph. For a subset $S \subset V$, we denote $\clP(S,t) \ceq \bigcup_{s \in S}\clP(s,t)$. 
Now we claim that there is a bijection between the tableaux $\QT^{(k,n-k)}_{\lambda/\mu}$ 
and the following set $\clP^{(k,n-k)}_0(\lambda/\mu)$ of non-intersecting lattice paths.
Let us define the vertices $u_i \ceq (\lambda_i,k+n)$ for $i=1,\dotsc,l$, 
and $v_j \ceq (\mu_j,0)$ for $j=1,\dotsc,m$. Then 
\begin{align*}
 \clP^{(k,n-k)}_0(\lambda/\mu) \ceq \{(P_1,\dotsc,P_l) \mid \text{non-intersecting}, \ 
  P_i \in \clP(v_i,u_i) \text{ for } i \le m, \ 
  P_i \in \clP(I,u_i) \text{ for } i>m\}.
\end{align*}

\begin{figure}[H]
\begin{minipage}{0.50\textwidth}
\centering
\begin{tikzpicture}
\draw [dotted,step=0.5cm] (-1.49,-1.49) grid (2.0,2.8);
\foreach \x in {0.0,0.5,1.0,1.5,2.0,2.5,3.0}
 \draw [dotted] (-1.5,-0.75+\x) -- (-1.0,-0.5+\x);
\foreach \x in {0.0,0.5,1.0,1.5,2.0,2.5}
 \draw [dotted] (-1.0,-0.50+\x) -- (2.0-\x, 2.5);
\foreach \x in {0.0,0.5,1.0,1.5,2.0,2.5,3.0}
 \draw [dotted] (-1.50+\x,-1.5) -- (2.0, 2.0-\x);
\draw (-1.0, 2.5) circle (0.5mm) node (u5) [above] {$u_5$};
\draw (-0.5, 2.5) circle (0.5mm) node (u4) [above] {$u_4$};
\draw ( 1.0, 2.5) circle (0.5mm) node (u3) [above] {$u_3$};
\draw ( 1.5, 2.5) circle (0.5mm) node (u2) [above] {$u_2$};
\draw ( 2.0, 2.5) circle (0.5mm) node (u1) [above] {$u_1$};
\draw (-1.0,-1.5) circle (0.5mm) node (v3) [below] {$v_3$};
\draw ( 0.5,-1.5) circle (0.5mm) node (v2) [below] {$v_2$};
\draw ( 1.5,-1.5) circle (0.5mm) node (v1) [below] {$v_1$};
\draw (-1.5, 2.0) circle (0.5mm); 
\draw (-1.5,-0.75)circle (0.5mm); 
\draw[->] (1.5,-1.45)--(1.5,-1.0); 
\foreach \x in {0.5,1.0,1.5,2.0}
 \draw[->] (1.5,-1.5+\x)--(1.5,-1.0+\x); 
\draw[->] (1.5,1.0)--(2.0,1.5); \draw (1.9,1.35) node [below] {$\ol{3}'$}; 
\draw[->] (2.0,1.5)--(2.0,2.0); \draw[->] (2.0,2.0)--(2.0,2.45);
\draw[->] (0.5,-1.45)--(0.5,-1.0); 
\foreach \x in {0.5,1.0}
 \draw[->] (0.5,-1.5+\x)--(0.5,-1.0+\x); 
\draw[->] (0.5,0.0)--(1.0,0.0); \draw (0.75,0.0) node [below] {$2$}; 
\foreach \x in {0.0,0.5,1.0}
 \draw[->] (1.0,0.0+\x)--(1.0,0.5+\x);
\draw[->] (1.0,1.5)--(1.5,1.5); \draw (1.25,1.5) node [above] {$\ol{3}$}; 
\draw[->] (1.5,1.5)--(1.5,2.0); \draw[->] (1.5,2.0)--(1.5,2.45); 
\draw[->] (-0.95,-1.45)--(-0.5,-1.0); \draw (-0.6,-1.15) node [below] {$1'$}; 
\foreach \x in {0.0,0.5,1.0,1.5}
 \draw[->] (-0.5,-1.0+\x)--(-0.5,-0.5+\x);
\draw[->] (-0.5,1.0)--(0.0,1.0); \draw (-0.25,1.0) node [below] {$3$};  
\draw[->] ( 0.0,1.0)--(0.5,1.0); \draw ( 0.25,1.0) node [below] {$3$};  
\foreach \x in {0.0,0.5,1.0}
 \draw[->] (0.5,1.0+\x)--(0.5,1.5+\x);
\draw[->] (0.5,2.5)--(.95,2.5); \draw (0.75,2.5) node [below] {$5$}; 
\draw[->] (-1.45,-0.73)--(-1.0,-0.5); \draw (-1.2,-0.6) node [below] {$\ol{1}'$}; 
\foreach \x in {0.0,0.5,1.0,1.5} 
 \draw[->] (-1.0,-0.5+\x)--(-1.0,0.0+\x);
\draw[->] (-1.0,1.5)--(-0.5,2.0); \draw (-0.65,1.85) node [below] {$4'$};  
\draw[->] (-0.5,2.0)--(-0.5,2.45);
\draw[->] (-1.45,2.0)--(-1.0,2.0); \draw (-1.25,2.0) node [below] {$4$};  
\draw[->] (-1.0,2.0)--(-1.0,2.45);
\fill (-1.5,-1.5) circle (0.5mm) node [left] {$(0,0)$};
\fill (-1.5, 1.5) circle (0.5mm) node [left] {$(0,2k)$};
\fill (-1.5, 2.5) circle (0.5mm) node [left] {$(0,n+k)$};
\end{tikzpicture}
\end{minipage}
\begin{minipage}{0.15\textwidth}
\centering
\ytableausetup{mathmode,smalltableaux}
\begin{ytableau} 
 {} & & & & & & *(gray) \\ 
 \none  & & & & & *(gray) & *(gray) \\ 
 \none  & \none & & *(gray) & *(gray) & *(gray) & *(gray) \\
 \none  & \none & \none & *(gray) & *(gray) \\
 \none  & \none & \none & \none   & *(gray)
\end{ytableau}
\end{minipage}
\hspace{0.05\textwidth}
\begin{minipage}{0.25\textwidth}
\centering
\ytableausetup{mathmode,nosmalltableaux}
\begin{ytableau} 
 {} & & & & & & \ol{3}' \\ 
 \none & & & & & 2 & \ol{3} \\ 
 \none & \none & & 1' & 3 & 3 & 5 \\
 \none & \none & \none & \ol{1}' & 4' \\
 \none & \none & \none & \none & 4
\end{ytableau}
\end{minipage}
\caption{Non-intersecting lattice paths in $\Gamma^{(k,n-k)}$ and the corresponding tableaux}
\label{fig:iQ:LGV}
\end{figure}
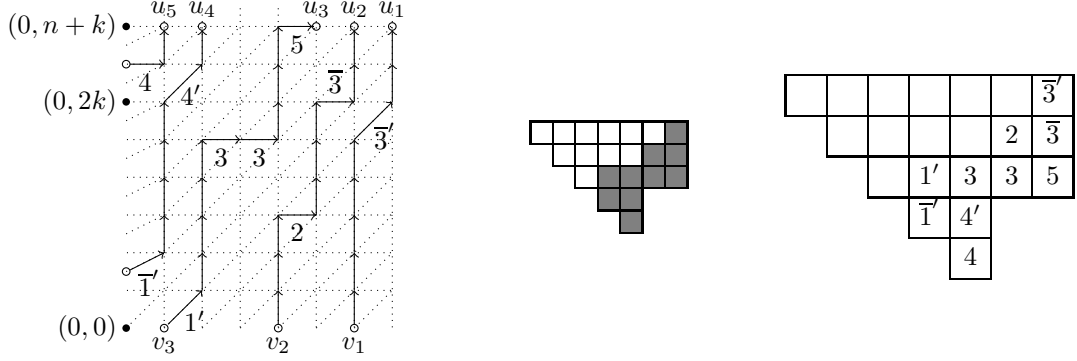

The bijection $\clP_0(\lambda/\mu) \to \QT^{(k,n-k)}_{\lambda/\mu}$ is given as follows. On each edge $e \in E$, we put the letter $l(e) \in \{1'<1<\ol{1}'<\ol{1}<\dotsb<\ol{k}<(k+1)'<k+1<\dotsb<n'<n\}$ (see \eqref{eq:iQ:k-n}) by the rule
\begin{align*}
 e \in H                  \colon &l\bigl(* \to (i,j)\bigr) \ceq \tfrac{j+1}{2}, & 
 e \in D \cup D_0         \colon &l\bigl(*\to(1,j)\bigr) \ceq (\tfrac{j+1}{2})', \\
 e \in \ol{H}             \colon &l\bigl(* \to (i,j)\bigr) \ceq \ol{\tfrac{j}{2}},  & 
 e \in \ol{D}\cup\ol{D}_0 \colon &l\bigl(*\to(1,j)\bigr) \ceq \ol{\tfrac{j}{2}}', \\
 e \in H'                 \colon &l\bigl(* \to (i,j)\bigr) \ceq j-2k, & 
 e \in D'                 \colon &l\bigl(* \to (1,j)\bigr) \ceq (j-2k)', 
\end{align*}
and  $l(e) \ceq \emptyset$ for $e \in P$. 
Then, given $(P_1,\dotsc,P_l) \in \clP^{(k,n-k)}_0(\lambda/\mu)$, we fill the first row of the skew diagram $S(\lambda/\mu)$ by the letters of the edges on $P_1$ from left to right, fill the second row by the letters on $P_2$, and so on. In \cref{fig:iQ:LGV}, we show an example of $(P_1,\dotsc,P_5) \in \clP_0^{(k,n-k)}(\lambda/\mu)$ with $k=3$, $n=2$, $\lambda=(7,6,5,2,1)$ and $\mu=(6,4,1)$. The shifted skew diagram $S(\lambda/\mu)$ consists of the gray boxes in the middle figure, and the tableau corresponding to the paths $(P_1,\dotsc,P_5)$ is shown in the right figure. 

Using the terminology in \cite{S}, we find that the set of vertices $\{u_1,\dotsc,u_l\}$ is $\Gamma^{(k,n-k)}$-compatible with the union $\{v_1,\dotsc,v_m\} \cup I$, and we can apply the Pfaffian formula in \cite[Theorem 3.2]{S} to obtain the statement.
\end{proof}

\begin{rmk}
We devised the direct graph $\Gamma^{(k,n-k)}$ as a $Q$-function analogue of \cite[Figure 2.1 in Proof of Proposition 2.3]{Oisp}, which was used by Okada to derive the Jacobi-Trudi type identity of intermediate symplectic Schur polynomials. In the case $k=0$, $\Gamma^{(0,n)}$ is nothing but the directed graph $D$ used by Stembridge in \cite[Theorem 6.2]{S} to show the J\'ozefiak-Pragacz Pfaffian formula of skew Schur's $Q$-polynomial. In the case $k=n=1$, $\Gamma^{(1,0)}$ is the directed graph $G$ used by Okada in  \cite[Fig.\ 1 in Proof of Lemma 4.4]{OsQ} to prove the determinant formula of skew symplectic $Q$-polynomial, which we already cited in the proof of \cref{lem:iQ:tab} \ref{i:lem:iQ:tab:3}.
\end{rmk}

%

\section{Concluding remarks and questions}

At this moment, we only know the properties of intermediate symplectic $Q$-polynomials given in \cref{lem:iQ:tab} and \cref{thm:iQ:iJP}. We expect several other properties from those for Schur's and symplectic $Q$-polynomials. Below we list them in the form of open questions.

\begin{q}
Is there a good formula of $Q^{(k,n-k)}_\lambda(x_1,\dotsc,x_n)$ with $\ell(\lambda)=2$?
Combined with \eqref{eq:iQ:l=1} and \cref{thm:iQ:iJP}, it will yield a recursive definition of $Q^{(k,n-k)}_\lambda$ for general $\lambda$. 
\end{q}

\begin{q}
Schur's $Q$-function $Q^A_\lambda(x_1,\dotsc,x_n)$ can be presented as a ratio of Pfaffians, known as Nimmo's formula \cite[(A13)]{N}. A similar formula is shown for symplectic $Q$-functions by Okada \cite[Proposition 2.2]{OsQ}. Is there a Nimmo-type formula for intermediate symplectic $Q$-functions?
\end{q}

\begin{q}
Okada established in \cite{OgPQ} a theory of generalized Schur's $P$- and $Q$-functions. Actually, the theory of symplectic $Q$-functions in \cite{OsQ} is based on that theory. Also, this theory can be regarded as a $Q$-function analogue of Macdonald's ninth variation of Schur polynomials \cite{M9}. Is it possible to further generalize the theory to include our intermediate symplectic $Q$-functions? Such a framework will give answers to the above questions automatically.
\end{q}

\begin{q}
We expect that the situation is quite simplified in the case $n-k=1$, i.e., when the number of type $A$ variable is one. Some of the questions above might be attacked in this case.
\end{q}

\begin{Ack}
We would like to thank Professor Soichi Okada for the explanation of the directed graphs and non-intersecting lattice paths used in \cite{OsQ,Oisp}, which helped the author to devise the graph $\Gamma^{(k,n-k)}$ in \cref{fig:iQ:LGV}.
\end{Ack}

\end{document}